\documentclass[12pt,a4paper]{amsart}
\usepackage{graphicx,amssymb}
\usepackage{amsmath,amssymb,amscd}
\usepackage{hyperref}

\newtheorem{theorem}{Theorem}[section]

\newtheorem{proposition}[theorem]{Proposition}

\newtheorem{corollary}[theorem]{Corollary}

\newtheorem{remark}[theorem]{Remark}

\newtheorem{lemma}[theorem]{Lemma}

\newtheorem{conjecture}[theorem]{Conjecture}

\newcommand{\ra}{\rightarrow}

\newcommand{\PP}{\mathbb P}

\newcommand{\cO}{\mathcal{O}}
\newcommand{\cL}{\mathcal{L}}

\newcommand{\cM}{\mathcal{M}}

\newcommand{\rk}{\mbox{rk}}

\newcommand{\Cl}{\mbox{Cliff}}

\numberwithin{equation}{section}

\begin{document}

\baselineskip=16pt

\title[Butler's conjecture]{On linear series and a conjecture of D. C. Butler}

\author{U. N. Bhosle, L. Brambila-Paz and P. E. Newstead}

\address{Department of Mathematics, Indian Institute of Science, Bangalore 560012, India}
\email{usha@math.tifr.res.in}

\address{CIMAT A.C., Jalisco S/N, Mineral de Valencia, 36240, Guanajuato, Guanajuato,
M\'exico}
\email{lebp@cimat.mx}

\address{Department of Mathematical Sciences, The University of Liverpool,
Peach Street, Liverpool L69 7ZL, UK}
\email{newstead@liv.ac.uk}

\keywords{Linear series, coherent systems, stability, Brill-Noether, Petri curve}

\subjclass[2010]{14H60}

\date{\today}
\thanks{All authors are members of the research group Vector Bundles on Algebraic Curves. The second author thanks ICTP for hospitality during the preparation of this work.}



\begin{abstract} Let $C$ be a smooth irreducible projective curve
of genus $g$ and $L$ a line bundle of degree $d$ generated by a
linear subspace $V$ of $H^0(L)$ of dimension $n+1$. We prove a conjecture of D. C. Butler on the semistability of the kernel of the evaluation map $V\otimes{\mathcal O}_C\to L$ and obtain new results on the stability of this kernel. The natural context for this problem is the theory of coherent systems on curves and our techniques involve wall crossing formulae in this theory.

\end{abstract}

\maketitle

\section{Introduction}\label{intro}

Our primary object in this paper is to prove a long-standing conjecture of D. C. Butler on the semistability of the kernels of evaluation maps of linear series on curves. We also extend substantially the known results on the stability of these kernels. 

Let $C$ be a smooth irreducible projective curve of genus $g$. Let
$(L,V)$ be a linear series of type $(d,n+1)$ with $L$ a generated
line bundle of degree $d$ on $C$ and $V\subset H^0(L)$ a linear
subspace of dimension $n+1$ which generates $L$. We suppose that
$n\ge2.$ 

The kernel of the evaluation map  $V \otimes\cO_C\ra L$,
denoted by $M_{L,V}$, fits in an exact sequence
\begin{equation}\label{eq0}
0\ra M_{L,V}\ra V \otimes\cO_C\ra L\ra0.
\end{equation}
The vector bundle $M_{L,V}$ is variously called a kernel bundle or a syzygy bundle (sometimes also a Lazarsfeld bundle). When $V=H^0(L)$ (in other words, the linear series is complete), the bundle $M_{L,H^0(L)}$ is often denoted by $M_L$ and can be seen as an integral transform of $L$ (see, for example, \cite{el}). The bundle $M_{L,V}$ and its dual $M^*_{L,V}$ have been extensively studied over many years, mainly because of their application in syzygy problems. In fact, the Koszul cohomology groups encode certain properties of $M_{L,V}$ and its exterior powers (see, for example, \cite[section 2.1]{an}) and are fundamental in the study of syzygies (see, among others, \cite{an} for a general account, \cite{voi1,voi2,te,te1} for Green's Conjecture and \cite{fmp} for the Minimal Resolution Conjecture). For an overview of early work in this area, see \cite{laz}. Note that the
sequence $(\ref{eq0})$ can also be seen as the pullback of the
(dual of the) Euler sequence via the morphism $\phi_{L,V}:C\to
\PP(V^*)=\PP^n$ defined by the subspace $V$ of $H^0(L)$, so that
$M_{L,V}\simeq \phi^*_{L,V}(\Omega_{\PP^n}(1))$. The natural context for the study of $M_{L,V}^*$ is that of coherent systems (see section \ref{back} for more details). The bundles $M_L$ have also been studied in connection with theta-divisors \cite{pop1,pop2,bea,mis,ms}.

One of the main questions about $M_{L,V}$ is when it is stable or semistable; this has major implications in higher rank Brill-Noether theory (see, for example, \cite[section 11]{bgmn}, \cite[section 9]{bbn}). In the complete case, $M_L$ is semistable (stable) if $d\ge2g$ ($d>2g$) (\cite[Proposition 3.2]{el}; see also \cite[Theorem 1.2]{bu1}). This result was extended to the range $d\ge2g- \Cl(C)$ in \cite{cc} (here $\Cl(C)$ denotes the Clifford index of $C$). For $C$ a general curve, $M_L$ is always semistable and precise conditions for stability are known \cite[Theorem 2]{bu} (see also \cite{sch}). An interesting result, not restricted to general curves, is that, if $L$ computes $\Cl(C)$, then $M_L$ is semistable (and even stable unless $C$ is hyperelliptic) (see \cite[Theorem 1.3]{le}, \cite[Corollary 5.5]{ms}).

Much less is known when $V\ne H^0(L)$. In \cite{bu}, D. C. Butler made a
conjecture \cite[Conjecture 2]{bu} for the more general case where
$L$ is replaced by a semistable vector bundle $E$ and $V$ is any linear subspace of $H^0(E)$ that generates $E$; for the case of a
line bundle, this can be stated in the following form (except that
Butler restricts to the case $g\ge3$).
\begin{conjecture}\label{conj1}For $C$ a general curve of genus $g\ge1$
and a general choice of $(L,V)$, the bundle $M_{L,V}$ is semistable.
\end{conjecture}

There are many variants of Conjecture \ref{conj1}, of which we consider
the following two (see section \ref{back} for the definition of a \emph{Petri} curve).
\begin{conjecture}\label{conj2}For $C$ a general curve of genus $g\ge3$
and a general choice of $(L,V)$, the bundle $M_{L,V}$ is stable.
\end{conjecture}
\begin{conjecture}\label{conj3}
For $C$ a Petri curve of genus $g\ge3$ and a general choice of $(L,V)$, the bundle $M_{L,V}$ is stable.
\end{conjecture}

Conjectures \ref{conj2} and \ref{conj3} can fail for $g\le2$. For
$g=0,1$, this is obvious; if $g=0$, there are no stable bundles of
rank $n\ge2$, while, for $g=1$, the same is true if $\gcd(n,d)>1$.
For $g=2$, see \cite[Th\'eor\`eme 2]{mer2} and \cite[Theorem
8.2]{bbn}. For all $g$, a sharp lower bound for the existence of
$(L,V)$ of type $(d,n+1)$ on a Petri curve is $d\ge
g+n-\left\lfloor\frac{g}{n+1}\right\rfloor$. 

Much work has been done on these conjectures, although none of
them has been completely solved \cite{mer1, bh, le, bbn, bo, mis, afo}.  Techniques used include the use
of deformations, classical Brill-Noether theory and coherent
systems (see section \ref{back} for details of some of the results
obtained). Our object in this paper is to put together the most
far reaching of these results and to extend the techniques (based on wall-crossing formulae for coherent systems) used in
our previous paper \cite{bbn}. This allows us to complete the proof of Conjecture \ref{conj1} and to prove stability in
more cases.

In section \ref{back}, we recall some definitions and state some
known results. In section \ref{coh}, we review the properties of coherent systems which we require for our proofs. Section \ref{conj} contains our key theorem, which extends the range of values for which Conjecture \ref{conj3} is known and forms the basis for our results on the conjectures. We start with some lemmas which improve on some of those in \cite[Section 6]{bbn} and which combine to give the following theorem.

\noindent{\bf Theorem \ref{th7}.}\begin{em}
Suppose that $C$ is a Petri curve, $g\ge2$, $n\ge2$ and
\begin{equation}\label{eq31}
g+n-\left\lfloor\frac{g}{n+1}\right\rfloor\le d<g+n+\frac{(n^2-n-2)g}{2(n-1)^2}.
\end{equation}
Then, for the general linear series $(L,V)$ of type $(d,n+1)$, $M_{L,V}$ is stable. When $g$ is odd, the inequality $<$ in the hypothesis can be replaced by $\le$.
\end{em}

In section \ref{butler}, we prove our main theorem (Butler's Conjecture):

\noindent{\bf Theorem \ref{th5}.}\begin{em}
Let $C$ be a general curve of genus $g\ge1$ and $(L,V)$ a general linear series of type $(d,n+1)$. Then $M_{L,V}$ is semistable.\end{em}

\noindent We also make some deductions concerning Conjecture \ref{conj2}.

It is proved in \cite{bbn} that, if $n\le4$, Conjecture \ref{conj3} is true for all allowable $(g,d,n)$. In section \ref{n=5}, we use the results in section \ref{conj} to prove the following theorem, which is a substantial improvement on the previously known result that Conjecture \ref{conj3} (hence also Conjecture \ref{conj2}) holds for $g\ge n^2-1$.

\noindent{\bf Theorem \ref{th6}.}\begin{em}
Let $C$ be a Petri curve of genus $g\ge3$, $n\ge5$ and $g\ge2n-4$. Suppose that $(L,V)$ is a general linear series of type $(d,n+1)$. Then $M_{L,V}$ is stable.\end{em}

As corollaries of this theorem (Corollaries \ref{th1}, \ref{th4} and \ref{cornew}), we present  short lists of possible exceptions to Conjecture \ref{conj3} for $n=5,6,7$, which enable us to prove Conjecture \ref{conj2} completely for $n=5$ and $n=7$. Finally, we obtain new information about Conjectures \ref{conj3} and \ref{conj2}, proving in particular

\noindent{\bf Theorem \ref{prime2}.}\begin{em}
Let $n$ be a prime number, $C$ a general curve of genus $g\ge3$ and $(L,V)$ a general linear series of type $(d,n+1)$. Then $M_{L,V}$ is stable except possibly in the following cases: 
\begin{itemize}
\item $n\ge11$, $n+2 \le g\le\min\{n+4, g_n\}$, $d=3n$

\item  $n\ge17$, $n+5\le g\le g_n$, $d=3n,4n$.
\end{itemize}
\end{em}

\noindent Here $g_n=\frac{4(n-1)^2}{3n-5}$ (see Remark \ref{gn}). 

We thank the referee for a careful reading of the paper and a couple of useful suggestions.

\section{Background}\label{back}
We suppose always that $C$ is a smooth projective curve of genus
$g\ge1$ defined over the complex numbers with canonical bundle $K$.
When we say that $C$ is \emph{general}, we mean that $C$ lies in
some unspecified non-empty Zariski-open subset of the moduli space
of curves of genus $g$. The curve $C$ is a \emph{Petri} curve if
the multiplication map
$$H^0(L)\otimes H^0(K\otimes L^*)\to H^0(K)$$
is injective for every line bundle $L$ on $C$. It is a standard
fact that Petri curves do define a non-empty Zariski-open subset
of the moduli space, so any result which is valid for Petri curves of genus $g$ is also valid for a general curve of genus $g$.
\begin{remark}\label{rem1}\begin{em}
It is known that, on a Petri curve of genus $g$, a linear series $(L,V)$ of
type $(d,n+1)$ exists if and only if
$$d\ge g+n-\left\lfloor\frac{g}{n+1}\right\rfloor.$$
Moreover, Conjecture \ref{conj3} holds whenever $g\ge3$ and
\begin{equation}\label{eq1}
g+n-\left\lfloor\frac{g}{n+1}\right\rfloor\le d\le g+n
\end{equation}
and also for any $d \ge g+n-\left\lfloor\frac{g}{n+1}\right\rfloor$ when $g\ge n^2-1$ (see \cite{bu,le}).
\end{em}\end{remark}

Butler stated his conjecture in terms of coherent systems, which will play a central r\^ole in this paper. For the moment, we simply give the definitions and enough detail to state the general form of Butler's conjecture. We will give more details in section \ref{coh}. 

A coherent system $(E,V)$ of type
$(r,d,k)$ on $C$ is a pair consisting of a vector bundle $E$ of
rank $r$ and degree $d$ and a linear subspace $V\subset H^0(E)$ of
dimension $k$. The coherent system $(E,V)$ is said to be
\emph{generated} if  the evaluation map $V\otimes\cO\to E$ is
surjective. For any real
number $\alpha$, a coherent system $(E,V)$ of type $(r,d,k)$ is $\alpha$-{\it stable} if, for any proper coherent subsystem $(E',V')$ of type $(r',d',k')$ of $(E,V)$, 
$$\frac{d'+\alpha k'}{r'}<\frac{d+\alpha k}r.$$
There exist moduli spaces $G(\alpha;
r,d,k)$ of $\alpha$-stable coherent systems of type $(r,d,k)$; if $k>0$, a
necessary condition for the non-emptiness of $G(\alpha;r,d,k)$ is
that $\alpha>0$. There are finitely many critical values
$0=\alpha_0<\alpha_1<\cdots<\alpha_L$ of $\alpha$; as $\alpha$
varies, the concept of $\alpha$-stability remains constant between
two consecutive critical values. We denote by $G_0(r,d,k)$ (resp.
$G_L(r,d,k)$) the moduli spaces corresponding to
$0<\alpha<\alpha_1$ (resp. $\alpha>\alpha_L$). 
It is easy to see that, if $(E,V)\in G_0(r,d,k)$, 
then $E$ is semistable; moreover, if $E$ is stable, then, 
for any linear subspace $V$ of dimension $k,$  $(E,V)\in G_0(r,d,k)$. 

For any generated coherent system $(E,V)$, we have an exact sequence, analogous to \eqref{eq0}:
\begin{equation}\label{eq21}
0\ra M_{E,V}\ra V\otimes\cO\ra E\ra0.
\end{equation}
Suppose that $h^0(E^*)=0$. Dualising \eqref{eq21}, we obtain a new generated coherent
system $D(E,V):=(M_{E,V}^*,V^*)$, called the \emph{dual span} of $(E,V)$.

Let $S_\alpha(r,d,k)$ denote the subset of $G_\alpha(r,d,k)$ consisting of generated $\alpha$-stable coherent systems, endowed with its natural structure as an open subscheme of $G_\alpha(r,d,k)$. Note that, if $(E,V)\in S_0(r,d,k)$, then $E$ is semistable, so $h^0(E^*)=0$ and $D(E,V)$ exists.
The original form of Butler's conjecture can now be stated as follows.

\begin{conjecture}\label{conj4} \cite[Conjecture 2]{bu} Let $C$ be a general
curve of genus $g\ge3$ and let $(E,V)$ be a general element of $S_0(r,d,r+n)$. Then $D(E,V)\in S_0(n,d,r+n)$ and the formula $(E,V)\mapsto D(E,V)$ defines a birational map $S_0(r,d,r+n)\  - \rightarrow S_0(n,d,r+n)$.
\end{conjecture}

\begin{remark}\label{gen}\begin{em}
By ``general element'', Butler means that $(E,V)$ belongs to some Zariski-open subset of $S_0(r,d,r+n)$ which is dense in every irreducible component. In our case, on any Petri curve (hence also on a general curve), the variety $G(1,d,n+1)$ of linear systems of type $(d,n+1)$ is smooth and irreducible by classical Brill-Noether theory whenever $d>g+n-\frac{g}{n+1}$, and the general element of this variety is generated. When $d=g+n-\frac{g}{n+1}$, there are finitely many linear systems  of the given type (the number is given by a formula of Castelnuovo (see \cite[Chap. V, formula (1.2)]{acgh}) and all are generated. In the latter case, we regard all $(L,V)$ as being general; this presents no problem, since also all $M_{L,V}$ are stable (see \cite[Theorem 2]{bu}, \cite{le}) and the formula $(L,V)\mapsto M_{L,V}^*$ defines a bijection from $G\left(1,g+n-\frac{g}{n+1},n+1\right)$ to 
$$\left\{E\left| E\mbox{ stable},\rk E=n,\deg E=g+n-\frac{g}{n+1}, h^0(E)=n+1\right.\right\}.$$
\end{em}\end{remark}

\begin{remark}\label{rem6}\begin{em}
The case $r=1$ of Conjecture \ref{conj4} is
equivalent to Conjecture \ref{conj1}. In the first place, $(L,V)$ is
$\alpha$-stable for all $\alpha>0$ and, by Remark \ref{gen}, the general $(L,V)$ is generated. Moreover,  if $D(L,V)\in G_0(n,d,n+1)$, then $M_{L,V}^*$ is semistable (hence
also $M_{L,V}$). Conversely, noting that $D(L,V)$ is generated and $h^0(M_{L,V})=0$, we see that any proper subsystem of $D(L,V)$ of type $(r',d',k')$ has $k'\le r'$. Hence, if $M_{L,V}^*$ is semistable, then $D(L,V)\in G(\alpha;n,d,n+1)$ for all $\alpha$ and in particular $D(L,V)\in S_0(n,d,n+1)$. Finally, if $(E,W)\in S_0(n,d,n+1)$, then we have an exact sequence
$$0\to (\det E)^*\to W\otimes\cO_C\to E\to0;$$
dualising and noting again that $h^0(E^*)=0$, we obtain $(E,W)=D(\det E,W^*)$.
\end{em}\end{remark}

We turn now to known results on Conjectures \ref{conj1} and
\ref{conj2}. A recent preprint \cite{afo} addresses Conjecture
\ref{conj1}. The following is a reformulation of \cite[Theorem 1.8]{afo}.
\begin{proposition}\label{prop1} Let $C$ be a
general curve of genus $g\ge1$ and $(L,V)$ a general linear series
of type $(d,n+1)$. Suppose $g=ns+t$ with $0\le t\le n-1$. If
\begin{equation}\label{eq2}
d\ge \max\{g+n+\min\{n-t,t-2\},g+n\},
\end{equation}
then $M_{L,V}$ is semistable.
\end{proposition}
\begin{proof}
By \cite[Theorem 1.8]{afo}, $M_{L,V}$ is semistable whenever
$$d+n\left\lfloor\frac{d}n\right\rfloor\ge2g+2n-2.$$
The left hand side of this inequality is an increasing function of $d$. One can easily check that, if $d=g+2n$, the inequality holds; hence it holds also for $d\ge g+2n$. Suppose now that $d=g+n+a$ with $0\le a\le n-1$ and write $g=ns+t$ as in the statement. The inequality now reduces to
$$a+n\left\lfloor\frac{t+a}n\right\rfloor\ge t-2.$$
If $a\ge n-t$, this inequality holds since $t\le n-1$. If $a<n-t$, the inequality is equivalent to $a\ge t-2$. This completes the proof.
\end{proof}

\begin{remark}\label{g=1,2}\begin{em}
For $g=1, 2$, the conclusion of Proposition \ref{prop1} is already known for any smooth curve and the bound given by \eqref{eq2} is sharp \cite[Theorems 8.1 and 8.2]{bbn}.
\end{em}\end{remark}

\begin{corollary}\label{cor0}  Let $C$ be a
general curve of genus $g\ge1$ and $(L,V)$ a general linear series
of type $(d,n+1)$. Suppose $g=ns+t$ with $s\ge0$ and either $0\le t\le\min\{3,n-1\}$ or $t=n-1$. Then $M_{L,V}$ is semistable.
\end{corollary}
\begin{proof} This follows at once from Proposition \ref{prop1} and Remarks \ref{rem1} and \ref{g=1,2}.
\end{proof}

\begin{corollary}\label{corint} Let $C$ be a
general curve of genus $g\ge1$ and $(L,V)$ a general linear series
of type $(an,n+1)$ for some integer $a$. Then $M_{L,V}$ is semistable.
\end{corollary}
\begin{proof} If $an<g+n-\left\lfloor\frac{g}{n+1}\right\rfloor$, there are no linear series of type $(an,n+1)$. If $g+n-\left\lfloor\frac{g}{n+1}\right\rfloor\le an\le g+n$, the result follows from Remarks \ref{rem1} and \ref{g=1,2}. If $an>g+n$ and we write $g=ns+t$ with $0\le t\le n-1$, then $an>n(s+1)+t$, hence $an\ge n(s+2)=g+n+n-t$. The result now follows from Proposition \ref{prop1}.
\end{proof}

An earlier paper \cite{bh} addresses Conjectures \ref{conj1} and \ref{conj2}.
\begin{proposition}\label{prop2}\cite[Proposition 1.6, Theorem 1.7 and Remark 1.10]{bh}
Let $C$ be a general curve of genus $g\ge2$ and $(L,V)$ a general
linear series of type $(d,n+1)$ with $n\geq 3$.
\begin{itemize}
\item If $u=\left\lceil\frac{g-1}{n+2}\right\rceil$ and $d\ge n(u+1)+1$, then $M_{L,V}$ is semistable.
\item If $u=\left\lceil\frac{g-2}{n+2}\right\rceil$ and
$d\ge\max\{n(u+2)+1,3n+2\}$,
then $M_{L,V}$ is stable.
\end{itemize}
\end{proposition}
 Note that, if $n\ge3$, the second case of Proposition \ref{prop2} applies to
 all $d\ge g+n+1$ whenever $g\ge n^2$ (which is slightly stronger
 than the condition $g\ge n^2-1$ mentioned above).

Teixidor \cite{te2} also has some results on Conjecture
\ref{conj2}. All of these results are based on deforming reducible
nodal curves and so do not specify the meaning of the term
``general curve''. On the other hand, our results in \cite{bbn}
assume only that $C$ is a Petri curve and address Conjecture
\ref{conj3}.  To extend the range of values for which Conjecture
\ref{conj3} (hence also Conjecture \ref{conj2}) is known, we used wall-crossing formulae for coherent
systems. In the following sections, we will exploit these techniques
further and obtain new results on all three conjectures, in particular proving Conjecture \ref{conj1} in all cases.

\section{Coherent Systems}\label{coh}

In this section, we summarise some facts about coherent systems which we will need in section \ref{conj}. Details may be found in \cite{bgmn}, \cite{bgmmn} and \cite{bbn}.

With the notation of section \ref{back}, note that, if $(E,V)\in G_L(r,d,k)$ and $E$ is stable, then $(E,V)\in G(\alpha;r,d,k)$ for all $\alpha>0$. We define
\begin{eqnarray*}
U(r,d,k)&:=&\{(E,V)\in G_L(r,d,k)|E \mbox{ is
stable}\}\\&=&\{(E,V)|E \mbox{ is stable and } (E,V)\in
G(\alpha;r,d,k)\mbox{ for all }\alpha>0\}.
\end{eqnarray*}

\begin{remark}\label{rem7}\begin{em}
By the argument in Remark \ref{rem6}, we note that, if $M_{L,V}$ is
stable, then $D(L,V)\in U(n,d,n+1)$.  It follows that, on any Petri curve, Conjecture \ref{conj3} is equivalent to
the assertion that $U(n,d,n+1)\ne\emptyset$ when $d\ge g+n-\left\lfloor\frac{g}{n+1}\right\rfloor$ \cite[Proposition
9.6]{bbn}. From this and \cite[Theorem 5.4]{bgmmn} (see also \cite{mer1, mer2}), it follows that, when $3\le g\le n$, Conjecture \ref{conj3} holds also when
\begin{equation}\label{2n}
g+n\le d\le2n.
\end{equation}
\end{em}\end{remark}

\begin{lemma}\label{lem5}
On any smooth curve $C$, suppose that $U(n,d,n+1)\ne\emptyset$. Then $U(n,d+cn,n+1)\ne\emptyset$ for any positive integer $c$.
\end{lemma}
\begin{proof}
(Compare \cite[Remark 2.2]{bbn}.) Suppose $(E,W)\in U(n,d,n+1)$. Let $\cL$ be any effective line bundle of degree $c$ and let $s$ be a non-zero section of $\cL$. Then $(E\otimes\cL,W\otimes s)\in U(n,d+cn,n+1)$.
\end{proof}

Taking account of Remark \ref{rem7}, we have the following immediate corollary.
\begin{corollary}\label{cor5}
Suppose that $C$ is a Petri curve of genus $g\ge3$, $n\ge2$ and $M_{L,V}$ is stable for the general linear series $(L,V)$ of type $(d,n+1)$ for all $d$ with
$$g+n-\left\lfloor\frac{g}{n+1}\right\rfloor\le d<g+n+b,$$
where $b+\left\lfloor\frac{g}{n+1}\right\rfloor>n-1$. Then $M_{L',V'}$ is stable for the general linear series $(L',V')$ of type $(d',n+1)$ for all $d'\ge g+n-\left\lfloor\frac{g}{n+1}\right\rfloor$.
\end{corollary}

We recall that every irreducible component $X$ of $G(\alpha;r,d,k)$ has dimension
$$\dim X\ge\beta(r,d,k):= r^2(g-1)+1-k(k-d+r(g-1)).$$
On a Petri curve, we have (see \cite[Theorem 3.1]{bbn}) the equality
$$\dim G_L(n,d,n+1)=\beta(n,d,n+1)=\beta(1,d,n+1)=g-(n+1)(n-d+g).$$
Moreover 
\begin{itemize}
\item $G_L(n,d,n+1)\ne\emptyset$ if and only if $\beta(n,d,n+1)\ge0$ or equivalently $d\ge g+n-\frac{g}{n+1}$;
\item $G_L(n,d,n+1)$ is irreducible if $\beta(n,d,n+1)>0$ and finite if $\beta(n,d,n+1)=0$;
\item the general element of $G_L(n,d,n+1)$ (every element when $\beta(n,d,n+1)=0$) is generated.
\end{itemize}  

In order to prove that $U(n,d,n+1)\ne\emptyset$, we need to show that the general $(E,W)\in G_L(n,d,n+1)$ has $E$ stable. To do this, we either show that 
$$\{(E,W)\in G_L(n,d,n+1)|E \mbox{ not stable}\}=\emptyset$$ 
or prove that 
\begin{equation}\label{eq35}
\dim\{(E,W)\in G_L(n,d,n+1)|E \mbox{ not stable}\}<\beta(n,d,n+1).
\end{equation}
If $C$ is Petri, we can assume that $(E,W)$ is generated and $h^0(E^*)=0$ \cite[Theorem 3.1]{bbn}. If $E$ is not stable, there exists a stable subbundle $E_1$ of $E$ with $\mu(E_1)\ge\mu(E)$ and hence an extension of
coherent systems
\begin{equation}\label{eq8} 0\ra (E_1,V_1)\ra(E,W)\ra(E_2,V_2)\ra0
\end{equation}
with $E_1$ stable and $(E_i,V_i)$ of type $(n_i,d_i,k_i)$. Note that $(E_2,V_2)$ is generated and $h^0(E_2^*)=0$. It follows that
\begin{equation}\label{eq11}
k_2\ge n_2+1,\ \  1+\frac{1}{n_2}\left(g-\left\lfloor\frac{g}{n_2+1}\right\rfloor\right)\le\mu(E_2)\le\mu(E)\le\mu(E_1).
\end{equation}
(Note that this varies from the set-up of \cite[(6.1)]{bbn} in that we assume that $E_1$ is stable rather than $E_2$.)

Following \cite[(9)]{bgmn} and \cite[(2.6)]{bbn}, we define
\begin{equation}\label{eq14}
C_{12}:=(k_1-n_1)(d_2-n_2(g-1))+n_2d_1-k_1k_2.
\end{equation}
(These numbers are crucial in obtaining wall-crossing formulae for coherent systems.) We write also $N_2$ for the kernel of the evaluation map $V_2\otimes\cO_C\ra E_2$. Taking account of \cite[(2.7) and Remark 2.7]{bbn}, the following is a special case of \cite[Lemma 2.6]{bbn}.

\begin{lemma}\label{lemc12}
Suppose that, for $i=1,2$, $(E_i,V_i)$ varies in a family depending on at most $\beta(n_i,d_i,k_i)$ parameters. Suppose further that
\begin{equation}\label{eq29}
C_{12}>h^0(E_1^*\otimes N_2\otimes K)
\end{equation}
for all $(E_i,V_i)$. Then the coherent systems $(E,W)\in G_L(n,d,n+1)$ arising as extensions \eqref{eq8} depend on at most $\beta(n,d,n+1)-1$ parameters.
\end{lemma} 

\begin{remark}\label{k2n2}\begin{em}
We shall want to apply Lemma \ref{lemc12} when $k_2=n_2+1$ and hence $k_1=n_1$. Since $(E_2,V_2)$ is generated and $h^0(E_2^*)=0$, we have $(E_2,V_2)\in G_L(n_2,d_2,n_2+1)$, so $(E_2,V_2)$ depends on at most $\beta(n_2,d_2,n_2+1)$ parameters.
Since $(E_1,V_1)$ is of type $(n_1,d_1,n_1)$, it must be generically generated, for
otherwise there would exist a subsystem $(F,W_1)$ with $\dim W_1>\rk F$, contradicting the fact that $(E,W)\in G_L(n,d,n+1)$. Since  $E_1$ is stable, it follows that $(E_1,V_1)\in G_L(n_1,d_1,n_1)$. To see this, note first that $V_1\otimes {\mathcal O}_C \subset E_1$. Hence for any subsystem $(F_1,W_1) \subset (E_1,V_1)$, we have $W_1\otimes {\mathcal O}_C \subset F_1$, so that dim $W_1 \le \rk F_1$. This together with the fact that $\mu(F_1) < \mu(E_1)$ implies that $(E_1, V_1)$ is $\alpha$-stable for all $\alpha>0$. 
Since $(E_1,V_1)\in G_L(n_1,d_1,n_1)$, $(E_1,V_1)$ depends on at most $\beta(n_1,d_1,n_1)$ parameters by \cite[Theorem 5.6]{bgmn}. Thus Lemma \ref{lemc12} does apply in this case.
\end{em}\end{remark}

\section{Key Theorem}\label{conj}
Let $C$ be a Petri curve and let $(E,W)$ be a generated coherent
system of type $(n,d,n+1)$  with $h^0(E^*)=0$. Then $(E,W)\in
G_L(n,d,n+1)$. We consider extensions \eqref{eq8} satisfying \eqref{eq11}.

First suppose that $n_2=1$. 

\begin{lemma}\label{lem1}
Suppose that $C$ is a Petri curve, $n\ge2$, $n_2=1$ and 
\begin{equation}\label{eq10}
\frac{d}{n}<g+1-\left\lfloor\frac{g}2\right\rfloor.
\end{equation}
Then no extension \eqref{eq8} exists satisfying \eqref{eq11}.
\end{lemma}
\begin{proof} If such an extension exists, then, by \eqref{eq11}, $E_2$ is a line bundle with  $d_2\ge g+1-\left\lfloor\frac{g}2\right\rfloor$. Using \eqref{eq11} again, this contradicts \eqref{eq10}.
\end{proof}

Next we suppose that $k_2\ge n_2+2$. The following lemma generalises \cite[Proposition 6.12]{bbn}.
\begin{lemma}\label{lem2}
Suppose that $C$ is a Petri curve, $n\ge3$ and
\begin{equation}\label{eq9}
d<g+n+\min\left\{\frac{(n^2-n-2)g}{2(n-1)^2},n+\frac{g}{2n-1}\right\}.
\end{equation}
Then there exists no extension \eqref{eq8} satisfying \eqref{eq11} with $k_2\ge n_2+2$.
\end{lemma}
\begin{proof} Suppose that an extension \eqref{eq8} with the stated properties exists. Since $(E_2,V_2)$ is generated and $k_2\ge n_2+2$, we can choose a
subspace $W_2$ of $H^0(E_2)$ of dimension $n_2+2$ which generates $E_2$. This yields an exact sequence
$$0\ra N_2'\ra W_2\otimes\cO\ra E_2\ra0,$$
where $N_2'$ has rank $2$ and $h^0(N_2')=0$. Dualising this sequence
and using the fact that $h^0(E_2^*)=0$, we have $h^0(N_2'^*)\ge
n_2+2$; moreover $\deg N_2'^*=d_2$. Following the proof of
\cite[Proposition 6.12]{bbn}, we consider three possibilities.
\begin{itemize}
\item[(i)] $h^0(L_1)\le1$ for every line subbundle $L_1$ of $N_2'^*$.
\item[(ii)] There exists an exact sequence
$$0\ra L_1\ra N_2'^*\ra L_2\ra0$$
with $h^0(L_2)=s, 2\le s\le n_2$.
\item[(iii)] There exists an exact sequence as in (ii) with $h^0(L_2)\ge n_2+1$, $h^0(L_1)\ge2$.
\end{itemize}

In case (i), \cite[Lemma 3.9]{pr} implies that $h^0(\det N_2'^*)\ge2n_2+1$, so, by Remark \ref{rem1},
\begin{equation}\label{eq12}
d_2\ge g+2n_2-\frac{g}{2n_2+1}.
\end{equation}
In case (ii),
$$d_2=\deg L_1+\deg L_2\ge g+s-1-\frac{g}{s}+g+n_2+1-s-\frac{g}{n_2+2-s}.$$ 
Thus $$d_2\ge2g + n_2-\frac{(n_2+2)g}{s(n_2+2-s)}\, .$$
The right hand side of this expression takes its minimum value at $s=2$ (and $s=n_2$), so
\begin{equation}\label{eq13}
d_2\ge 2g+n_2-\frac{(n_2+2)g}{2n_2}.
\end{equation}
Finally, in case (iii),
$$d_2=\deg L_1+\deg L_2\ge g+1-\frac{g}2+g+n_2-\frac{g}{n_2+1}>2g+n_2-\frac{(n_2+2)g}{2n_2}.$$
So \eqref{eq13} holds again in this case.

Now recall from \eqref{eq11} that $\frac{d}n\ge\frac{d_2}{n_2}$; moreover $n_2\le n-1$. Substituting in \eqref{eq12} and \eqref{eq13}, we obtain a contradiction to \eqref{eq9}.
\end{proof}

\begin{remark}\label{rem8}\begin{em}
The second term in the minimum of \eqref{eq9} is essentially irrelevant for our purposes since $n+\frac{g}{2n-1}>n-1$ (see Corollary \ref{cor5}).
\end{em}\end{remark}

\begin{remark}\label{rem=}\begin{em}
The strict inequality $d<g+n+\frac{(n^2-n-2)g}{2(n-1)^2}$ is needed only in case (ii) of the proof of Lemma \ref{lem2} and only when all the inequalities are equalities. This requires in particular that $s=2$ or $s=n_2$ and that $g$ is even.
\end{em}\end{remark}

We turn now to the case $k_2=n_2+1$ with $n_2\ge2$ and prove a lemma generalising \cite[Proposition 6.10]{bbn}.

\begin{lemma}\label{lem4} Suppose that $C$ is a Petri curve of genus $g\ge2$. Then the extensions \eqref{eq8} satisfying
\eqref{eq11} with $n_2\ge2$, $k_2=n_2+1$ depend on at most
$\beta(n,d,n+1)-1$ parameters.
\end{lemma}
\begin{proof} In view of Lemma \ref{lemc12} and Remark \ref{k2n2}, it is sufficient to prove \eqref{eq29}. Note first that, by \eqref{eq14} and \eqref{eq11},
\begin{equation}\label{eq30}
C_{12}=n_2d_1-n_1(n_2+1)\ge n_1(d_2-n_2-1)\ge n_1\left(g-1-\left\lfloor\frac{g}{n_2+1}\right\rfloor\right)>0
\end{equation}
for $g\ge2$ and $n_2\ge2$.
So \eqref{eq29} holds if $h^0(E_1^*\otimes N_2\otimes K)=0$, where we recall that $N_2$ denotes the kernel of the evaluation map $V_2\otimes{\mathcal O}_C\to E_2$ and, in the present case, is a line bundle. If $h^0(E_1^*\otimes N_2\otimes K)>0$, then, since $E_1^*\otimes N_2\otimes K$ is stable, we have
$$0\le\deg(E_1^*\otimes N_2\otimes K)=-d_1-n_1d_2+n_1(2g-2)<n_1(2g-2)$$
and, by Clifford's Theorem for vector bundles,
$$h^0(E_1^*\otimes N_2\otimes K)\le n_1(g-1)-\frac{d_1+n_1d_2}2+n_1=n_1g-\frac{d_1+n_1d_2}2.$$
So, by \eqref{eq14}, it is sufficient to prove that 
$$n_2d_1-n_1(n_2+1)>n_1g-\frac{d_1+n_1d_2}2,$$
or equivalently
$$(2n_2+1)d_1+n_1d_2>2n_1(g+n_2+1).$$
Substituting $d_i\ge n_i\left(1+\frac{1}{n_2}\left(g-\left\lfloor\frac{g}{n_2+1}\right\rfloor\right)\right), i=1,2$, we have 
\begin{eqnarray*}
(2n_2+1)d_1+n_1d_2 &
\ge & \left((2n_2+1)n_1+ n_1 n_2\right) \left(1+\frac{1}{n_2}\left(g-\left\lfloor\frac{g}{n_2+1}\right\rfloor\right)\right) \\
&\ge & (3n_2+1)n_1\left(1+\frac{1}{n_2}\left(g-\frac{g}{n_2+1}\right)\right)  \\
&= & \frac{(3n_2+1)n_1(n_2+1+g)}{n_2+1}>2n_1(n_2+1+g)
\end{eqnarray*}
for $n_2\ge 2$. 
 This completes the proof of \eqref{eq29} and hence of the lemma.
\end{proof}

\begin{lemma} \label{min}
For $n\ge 2$,
$$\frac{(n^2-n-2)g}{2(n-1)^2}\le(n-1)g-n \left\lfloor\frac{g}{2}\right\rfloor,$$
with equality if and only if $n=2$ or $3$ and $g$ is even.
\end{lemma}
\begin{proof}
One has 
\begin{eqnarray*}
(n-1)g-n \left\lfloor\frac{g}{2}\right\rfloor - \frac{n^2-n-2}{2(n-1)^2}g &\ge& (n-1)g-\frac{ng}{2} - \frac{(n^2-n-2)g}{2(n-1)^2}\\& =& \frac{n(n-2)(n-3)g}{2(n-1)^2} \ge0.
\end{eqnarray*}	
Equality holds if and only if both inequalities are equalities. This proves the lemma.

\end{proof}

We can now prove our key theorem, which extends the range of values of $(g,n,d)$ for which Conjecture \ref{conj3} holds.

\begin{theorem}\label{th7}

Suppose that $C$ is a Petri curve, $g\ge2$, $n\ge2$ and
\begin{equation}\label{eq32}
g+n-\left\lfloor\frac{g}{n+1}\right\rfloor\le d< g+n+\frac{(n^2-n-2)g}{2(n-1)^2}.
\end{equation}
Then, for the general linear series $(L,V)$ of type $(d,n+1)$, $M_{L,V}$ is stable. When $g$ is odd, the inequality $<$ in the hypothesis can be replaced by $\le$.
\end{theorem}

\begin{proof}
If $n=2$, we must have $n_2=1$ in the extension \eqref{eq8}, so the result follows from Lemmas \ref{lem1} and \ref{min}. If $n\ge3$, then, in view of Lemma \ref{min} and Remarks \ref{rem7} and  \ref{rem8}, the first statement follows from \eqref{eq1} and Lemmas \ref{lem1}, \ref{lem2} and \ref{lem4}. The second statement follows from Remark \ref{rem=}.
\end{proof}

\begin{remark}\label{remge5}\begin{em}
Since Conjecture \ref{conj3} is known to hold for $n\le4$ (see \cite[section 7]{bbn}), we need this theorem only for $g\ge3$, $n\ge5$. However, the theorem as it stands provides a simpler proof of Conjecture \ref{conj3} for $n=3,4$ than the one in \cite{bbn}. In fact, for $n=3,g\ge4$ and $n=4,g\ge5$, Theorem \ref{th7} and Corollary \ref{cor5} give the conjecture directly. Using also \cite[Propositions 6.6 and 6.8]{bbn}, we are left only with the case $(g,n,d)=(3,4,10)$. This is a particular case of \cite[Proposition 7.6]{bbn}.
\end{em}\end{remark}

\section{Butler's Conjecture}\label{butler}

In this section we prove Butler's Conjecture (Conjecture \ref{conj1}).

\begin{theorem}\label{th5}
Let $C$ be a general curve of genus $g\ge1$ and $(L,V)$ a general linear series of type $(d,n+1)$. Then $M_{L,V}$ is semistable.
\end{theorem}

\begin{proof}
For $g=1,2$, see Remark \ref{g=1,2}. For the rest of the proof, we suppose $g\ge3$.

Let $g=ns+t$, where $0\le t\le n-1$, and let $(L,V)$ be a general linear series of type $(d,n+1)$. Then
$$\min\{n-t,t-2\}\le\left\lfloor\frac{n-2}2\right\rfloor.$$ 
By \eqref{eq2}, $M_{L,V}$ is semistable whenever $d\ge g+n+\left\lfloor\frac{n-2}2\right\rfloor$. In view of \eqref{eq1}, it follows that $M_{L,V}$ is semistable for all allowable $(g,d,n)$ with $n\le5$. By Theorem \ref{th7}, it is therefore sufficient to  prove that, when $n\ge6$,
\begin{equation}\label{eq26}
\frac{(n^2-n-2)g}{2(n-1)^2}>\frac{n-4}2.
\end{equation}

Since $(n-4)(n-1)^2=(n-5)(n^2-n-2)+6n-14$,  \eqref{eq26} is equivalent to
$$g>n-5+\frac{6n-14}{n^2-n-2}.$$
It is easy to check that $\frac{6n-14}{n^2-n-2}\le1$ for $n\ge4$, so \eqref{eq26} holds for $g\ge n-3$. 

On the other hand, if $3\le g<n$, $M_{L,V}$ is semistable for $g+n\le d\le 2n$ by Remark \ref{rem7}. By Proposition
\ref{prop1}, this holds also when $d\ge g+n+n-t=2n$ (as $g=t$ for $g<n$) and
hence for any $d\geq g+n-\left\lfloor\frac{g}{n+ 1}\right\rfloor=g+n$.
\end{proof}

We have the following immediate corollary.

\begin{corollary}\label{coprime}
Under the hypotheses of Theorem \ref{th5}, suppose further that $\gcd(n,d)=1$. Then $M_{L,V}$ is stable.
\end{corollary}

\begin{corollary}\label{prime}
Under the hypotheses of Theorem \ref{th5}, suppose further that $n$ is a prime integer and $3\le g\le n+1$. Then $M_{L,V}$ is stable. 
\end{corollary}
\begin{proof}
By Remark \ref{rem7} and Lemma \ref{lem5}, if $3\le g\le n$, $M_{L,V}$ is stable for $g+n \le d \le 2n$ and hence for $d = rn, r\ge 2, r$ being an integer. By Theorem \ref{th5}, $M_{L,V}$ is semistable for $d\ge g+n$. If $d$ is not a multiple of $n$, then $d$ is coprime to $n$ (as $n$ is prime). It follows that $M_{L,V}$ is stable in this case as well.

When $g=n+1$, $M_{L,V}$ is stable for $d=g+n-\left\lfloor\frac{g}{n+1}\right\rfloor=2n$ by Remark \ref{rem1} (see \eqref{eq1}) and the argument above still works. 
\end{proof}

\section{Conjectures \ref{conj2} and \ref{conj3}}\label{n=5}

We know from \cite{bbn} (see also Remark \ref{remge5}) that Conjecture \ref{conj3} holds for $n\le4$. In this section we consider larger values of $n$. We obtain also further results on Conjecture \ref{conj2}.

\begin{theorem}\label{th6}
Let $C$ be a Petri curve of genus $g\ge3$, $n\ge5$ and $g\ge2n-4$. Suppose that $(L,V)$ is a general linear series of type $(d,n+1)$. Then $M_{L,V}$ is stable.
\end{theorem}
\begin{proof}
In view of Theorem \ref{th7} and Corollary \ref{cor5}, it is  sufficient to  prove that, whenever $n\ge5$ and $g\ge2n-4$,
\begin{equation}\label{eq28}
\frac{(n^2-n-2)g}{2(n-1)^2}+\left\lfloor\frac{g}{n+1}\right\rfloor>n-1.
\end{equation}
For any fixed $n$, the left hand side of \eqref{eq28} is an increasing function of $g$. It is therefore sufficient to prove \eqref{eq28} when $g=2n-4$. This is a straightforward calculation.
\end{proof}

\begin{remark}\label{rem9}\begin{em} The theorem holds also for $g\ge2n-5$ if $n\ge7$. For $n\ge8$, $g=2n-5$, the inequality \eqref{eq28} is still valid. For $n=7$, $g=2n-5=9$, \eqref{eq28} becomes an equality; this is sufficient by Theorem \ref{th7} and Corollary \ref{cor5} since $g$ is odd. For $g=2n-6$, \eqref{eq28} fails for all $n$.\end{em}\end{remark}

\begin{corollary}\label{th1}
Suppose that $n=5$, $g\ge3$ and $d\ge g+5-\left\lfloor\frac{g}6\right\rfloor$. 
\begin{itemize}
\item If $C$ and $(L,V)$ are general, then $M_{L,V}$ is stable.

\item If $C$ is Petri and $(L,V)$ is general, then $M_{L,V}$ is stable except possibly in the following cases:

$g=3, d=12$

$g=4, d=12, 13, 17, 18$

$g=5, d=13, 18$.

\end{itemize}
\end{corollary}

\begin{proof} Since $2n-4=6=n+1$, the result for $C$ general holds by Theorem \ref{th6} and Corollary \ref{prime}.

Now suppose that $C$ is a Petri curve and $(L,V)$ is general. Then, by Theorem \ref{th6}, $M_{L,V}$ is stable for $g\ge6$.  For $g\le5$, all the cases except those in the list are covered by Theorem \ref{th7}, Lemma \ref{lem5} and \cite[Propositions 6.6 and 6.8]{bbn}.
\end{proof}

\begin{remark}\label{rem2}
\begin{em}By Lemma \ref{lem5}, to complete the proof of Conjecture \ref{conj3} for $n=5$, it is sufficient to prove the following cases: $g=3, d=12$; $g=4, d=12,13$; $g=5, d=13$.
\end{em}\end{remark}

\begin{corollary}\label{th4}
Suppose that $n=6$, $g\ge3$ and $d\ge g+6-\left\lfloor\frac{g}7\right\rfloor$. 
\begin{itemize}
\item If $C$ is Petri and $(L,V)$ is general, then $M_{L,V}$ is stable except
possibly in the following cases:

$g=3, d=14$

$g=4, d=14, 15, 20, 21$

$g=5, d=14, 15, 16, 20, 21, 22$

$g=6, d=16, 22, 28$.

\item If $C$ and $(L,V)$ are general, then $M_{L,V}$ is stable except possibly when $g=3,d=14$; $g=4, d=14, 15$; $g=5, d=14, 15, 16$ or $g=6, d=16$.

\end{itemize}
\end{corollary}

\begin{proof}
The argument is similar to that of Theorem \ref{th1}. The result holds for $g\ge8$ by Theorem \ref{th6}. We again argue case by case for $g\le7$.

The result for $C$ general follows from Proposition \ref{prop2}.
\end{proof}

\begin{remark}\label{rem3}\begin{em}
To complete the proof of Conjecture \ref{conj3} for $n=6$, it is sufficient to prove the cases $g=3, d=14$; $g=4, d=14, 15$; $g=5, d=14, 15, 16$; $g=6, d=16$.
\end{em}\end{remark}

\begin{corollary}\label{cornew}
Suppose that $n=7$, $g\ge3$ and $d\ge g+7-\left\lfloor\frac{g}8\right\rfloor$. 
\begin{itemize}
\item If $C$ is Petri and $(L,V)$ is general, then $M_{L,V}$ is stable except possibly in the following cases:

$g=3, d=16$

$g=4, d=16,17,23,24$

$g=5, d=16,17,18,23,24,25$

$g=6, d=17,18,19,24,25,26,31,32,33$

$g=7, d=18,19,25,26,32,33$.

\item If $C$ and $(L,V)$ are general, then $M_{L,V}$ is stable.
\end{itemize}
\end{corollary}
\begin{proof}
The result for $C$ Petri holds for $g\ge9$ by Theorem \ref{th6} and Remark \ref{rem9}. We argue case by case for $g\le8$ as before.

The result for $C$ general follows by Corollary \ref{prime} for $g\le8$ and by the first part of the proof for $g\ge9$.
\end{proof}

We could continue with higher values of $n$, but these would become increasingly complicated. However, we can obtain some simple statements, which lead to improved results for Conjecture \ref{conj2}, especially when $n$ is prime. We begin with a remark.

\begin{remark}\label{gn}\begin{em}
 Let 
$$g_n:= \frac{4(n-1)^2}{3n-5}\, .$$ 
Then $g_n \ge n+5$ if and only if $n\ge 17$. 
In fact, it is easy to see that $g_n \ge n+5$ if and only if $n(n-18)+29\ge 0$. The last inequality holds if and only if $n\ge 17$. Note also that $g_n\ge n+1$ for all $n$ and
\begin{equation}\label{eq34}
g>g_n\Longleftrightarrow g+n+\frac{(n^2-n-2)g}{2(n-1)^2}>3n.
\end{equation}
\end{em}
\end{remark}

\begin{proposition}\label{propgn}
Let $C$ be a Petri curve of genus $g>g_n$ with $n\ge5$ and let $(L,V)$ be a general linear series of type $(d,n+1)$. Then $M_{L,V}$ is stable for all $d\ge g+n-\left\lfloor\frac{g}{n+1}\right\rfloor$ with $d=rn$ for some integer $r$. Moreover, if $r_0$ is the smallest such integer, then $M_{L,V}$ is also stable for $d=r_0n+1$ and, if $r\ge r_0+1$, for $d=rn\pm1$.
\end{proposition}

\begin{proof}
If $g\ge2n-4$, the result follows immediately from Theorem \ref{th6}. We can therefore assume that $g\le2n-5$. By Remark \ref{gn}, we have also $g>n+1$ and hence
$$2n+1\le g+n-\left\lfloor\frac{g}{n+1}\right\rfloor<3n.$$
So the first value of $d=rn$ in the allowable range is given by $r=3$. Moreover, by \eqref{eq34} and Theorem \ref{th7}, $M_{L,V}$ is stable for $d=3n$. The proposition now follows from Lemma \ref{lem5} and \cite[Proposition 6.8]{bbn}.
\end{proof}

\begin{remark}\label{remnew}\begin{em}
For $n\ge9$, the condition $g>g_n$ is weaker than the hypotheses on $g$ in Theorem \ref{th6} and Remark \ref{rem9}.
\end{em} \end{remark}

\begin{proposition}\label{corgn}
Let $C$ be a general curve of genus $g>g_n$ and let $(L,V)$ be a general linear series of type $(d,n+1)$. Then $M_{L,V}$ is stable except possibly when 
$\gcd(n,d)>1$ and
$$n\ge11,\ \ \max\{n+4,g_n\}<g\le2n-6,\ \  g+n+\frac{(n^2-n-2)g}{2(n-1)^2}\le d\le g+2n-2.$$
\end{proposition}

\begin{proof}
For $n\le4$, we know already that $M_{L,V}$ is stable. For $n\ge5$, we can assume by Theorem \ref{th6} that $g\le2n-5$.  If $n=5,6$, then $g_n>2n-5$, contradicting the hypothesis that $g>g_n$. Hence we may assume that $n\ge7$ and, by Remark \ref{rem9}, $g\le2n-6$. It now follows from \eqref{eq34}, Theorem \ref{th7} and \cite[Proposition 6.8]{bbn} that $M_{L,V}$ is stable when
\begin{equation}\label{eq36}
g+n-\left\lfloor\frac{g}{n+1}\right\rfloor\le d\le 3n+1.
\end{equation}
If $g\le n+4$, then, by Proposition \ref{prop2}, $M_{L,V}$ is stable also when $d\ge 3n+2$. By Theorem \ref{th7}, it remains to consider the case 
$$n+5\le g\le2n-6,\ \ d\ge g+n+\frac{(n^2-n-2)g}{(n-1)^2}.$$
Lemma \ref{lem5} and \eqref{eq36} imply that $M_{L,V}$ is stable for
$$g+2n-1\le d\le 4n+1.$$
By Proposition \ref{prop2}, $M_{L,V}$ is stable for $d\ge4n+1$. 

Note finally that $n+5\le g\le 2n-6$ implies that $n\ge11$. This completes the proof.

\end{proof}

\begin{theorem} \label{prime2}
Let $n$ be a prime number, $C$ a general curve of genus $g\ge3$ and $(L,V)$ a general linear series of type $(d,n+1)$. Then $M_{L,V}$ is stable except possibly in the following cases:  
\begin{itemize}
\item $n\ge11$, $n+2 \le g\le\min\{n+4, g_n\}$, $d=3n$

\item  $n\ge17$, $n+5\le g\le g_n$, $d=3n,4n$.
\end{itemize}
\end{theorem}     

\begin{proof}
For $n$ prime, $n\le7$, this has already been proved. So suppose $n$ is prime, $n\ge11$. By Theorem \ref{th5}, $M_{L,V}$ is semistable for $d\ge g+n- \left\lfloor\frac{g}{n+1}\right\rfloor$. If $d$ is not a multiple of $n$, then $d$ is coprime to $n$ (as $n$ is prime), hence it suffices to check that $M_{L,V}$ is stable for all possible $d= rn, r$ a positive integer.

By Corollary \ref{prime}, Theorem \ref{th6}, Remark \ref{rem9} and Proposition \ref{propgn}, 
we can assume that $g\le g_n$ and $n+2 \le g \le 2n-6$. Then $2n+1 \le g+n-\left\lfloor\frac{g}{n+1}\right\rfloor \le 3n-7$ so that the least possible $d$ of the form $d=rn$ has $r=3$. Moreover, by Proposition \ref{prop2}, $M_{L,V}$ is stable for $d \ge 3n+2$ if 
$n+2 \le g\le n+4$ and for $d \ge 4n+1$ if $n+5 \le g \le 2n-6$. Noting that $g_n<2n-6$ for $n\ge9$, the result now follows.
\end{proof}

\end{document}